\documentclass[11pt]{amsart}
\usepackage{geometry}                
\usepackage{graphicx}
\usepackage{amssymb}
\usepackage{epstopdf}
\usepackage{xypic}
\usepackage{pstricks, pst-node, pst-tree}

\newtheorem{proposition}{Proposition}
\newtheorem{corollary}{Corollary}
\newtheorem{theorem}{Theorem}
\newtheorem{lemma}{Lemma}
\newtheorem{conjecture}{Conjecture}
\newtheorem{question}{Question}

\theoremstyle{definition}
\newtheorem{definition}{Definition}
\newtheorem{example}{Example}
\newtheorem{remark}{Remark}

\newcommand{\calC}{\mathcal C}

\newcommand{\NN}{\mathbb N}

\newcommand{\Hom}{\text{\rm Hom}}
\newcommand{\End}{\text{\rm End}}

\title{Products in a Category with One Object}
\author{Aaron Gray}
\address{National Security Agency}
\author{Keith Pardue}
\address{National Security Agency}
\email{pardue@member.ams.org}

\date{18 August 2016}                                           

\begin{document}

\begin{abstract}
We study monoids equipped with a second binary operation that captures the structure of the endomorphisms of an object $X$ such that $X=X\times X$. We construct a universal monoid of this type and examine some of its rich combinatorial structure. We show that if $X$ has a nontrivial endomorphism and $X=X\times X$, then every finite monoid has a faithful action on $X$.
\end{abstract}

\maketitle
\section{Introduction}

In a category $\calC$, consider an object $X$ such that $X$ is (isomorphic to) $X\times X$. In familiar categories, such objects are either very trivial or very large. In the category of sets, the sets with this property have $0$, $1$ or infinitely many elements. A vector space with this property has $0$ or infinite dimension. To say that $X$ is a product of two copies of itself is to say that there are two distinguished elements $\pi_1,\pi_2\in\End(X)$ such that for any object $Y$, the function
$$\Pi_Y:\Hom(Y,X)\rightarrow\Hom(Y,X)\times\Hom(Y,X)$$
given by $\Pi_Y(f)=(\pi_1f,\pi_2f)$ is a bijection. 
The inverse bijection $\Sigma:\Hom(Y,X)^2\rightarrow\Hom(Y,X)$ satisfies the identities $\pi_1\Sigma(a,b)=a$ and $\pi_2\Sigma(a,b)=b$ for all $a,b\in\Hom(Y,X)$. 
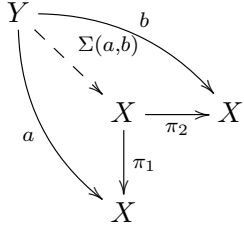
\[\xymatrix{
Y\ar@{-->}[dr]^{\Sigma(a,b)}\ar@/_1pc/[ddr]_a\ar@/^1pc/[rrd]^b &&\\
& X\ar[d]^{\pi_1}\ar[r]_{\pi_2} & X\\
& X &
}\]
The object $X$ retains this property in the full subcategory of $\calC$ in which $X$ is the only object, and the monoid $\End(X)=\Hom(X,X)$ is the set of all morphisms. Conversely, we make the following definition.

\begin{definition}\label{Def:CP} A \emph{categorical product (CP)} monoid $(M,\Sigma,\pi_1,\pi_2)$ is a monoid $M$ with distinguished elements $\pi_1,\pi_2\in M$ and a bijective function $\Sigma:M^2\rightarrow M$ such that $\pi_1\Sigma(a,b)=a$ and $\pi_2\Sigma(a,b)=b$ for all $a,b\in M$.\end{definition}

From a CP monoid we may make a category with one (arbitrary and redundant) object $X$ for which $M$ is the monoid of endomorphisms. In this category, $X=X\times X$. We give an alternative characterization of CP monoids in Proposition \ref{Prop:CPalt}. Example \ref{Ex:VS} in the next section is an example of a CP monoid in the spirit of the discussion above.

The main result of this paper is a construction of a universal CP monoid $U$ in Definition \ref{Def:U} and Theorem \ref{thm:Universal}. This monoid has a rich combinatorial structure that has surprised and delighted us, as we were led to it by purely abstract speculations about the properties described above.

Here is an outline of the paper. In Sections 2 and 3, we introduce basic notions. Section 2 concerns categorical quasiproduct (CQP) monoids. These satisfy a relaxation of the CP condition and are important to our study of CP monoids. Section 3 has basic facts about free magmas in the form that we use them in this paper. The heart of the paper is Section 4 in which we construct the universal CQP monoid $T$ and the universal CP monoid $U$; these are initial objects in their respective categories. In Section 5 we characterize left and right invertible elements of $T$ and $U$ and then give a bijection between units in a CP monoid $M$ and the different CP structures on $M$. In Section 6 we show that if $M$ is a non-zero CP monoid, then \emph{every} finite monoid $N$ has an injective homomorphism into $M$. This implies that if $X$ is an object of some category such that $X=X\times X$, then either the only endomorphism of $X$ is the identity, or every finite monoid $N$ has a faithful action on $X$. This makes precise the dichotomy between trivial objects that are products of themselves, and very large objects that are products of themselves. In the final section, we give some open problems.

We could consider all of these questions for coproducts as well as products, and indeed everything in this paper translates to that setting in the most straightforward way. Since coproducts are products in the opposite category, we need only replace monoids by their opposite monoids and swap left and right actions to obtain the parallel constructions and theorems for categorical coproduct monoids. We will not address coproducts explictly any further.

S. Pardue, R. Rodriguez, and M. Sweedler all gave us helpful comments that guided parts of this paper, and we are grateful for their help.

In this paper, we develop several particular monoids, a homomorphism and a congruence that we repeatedly refer to throughout the paper. We close this introduction with an index so that the reader can easily find the descriptions of these objects.

\begin{itemize}
\item $B$: The \emph{branch} monoid introduced in Example \ref{Ex:FBS}.
\item $F$: A free monoid on $\pi_1$ and $\pi_2$ introduced in Example \ref{Ex:FBS}.
\item $S$: A semiring built from subsets of $B$ and whose multiplicative structure is a CQP monoid introduced in Example \ref{Ex:FBS}.
\item $T$: The universal CQP monoid constructed in the beginning of Section \ref{Section:Universal}.
\item $U$: The universal CP monoid defined as a quotient of $T$ in Definition \ref{Def:U}.
\item $\beta$: The branch homomorphism from $T$ to $S$ defined in Definition \ref{Def:Beta}.
\item $\equiv$: A congruence on $T$ whose quotient is $U$, defined in Definition \ref{Def:Congruence}.
\end{itemize}

\section{Categorical Quasiproduct Monoids}\label{Section:CQPMonoids}

In this paper, we write all monoids multiplicatively with identity $1$. If $M$ has just one element, then we write $M=0$. 

\begin{proposition}\label{Prop:CPalt} Let $M$ be a monoid, $\pi_1,\pi_2\in M$ and $\Sigma:M^2\rightarrow M$ be a function. Then $(M,\Sigma,\pi_1,\pi_2)$ is a CP monoid if and only if $\Sigma,\pi_1$ and $\pi_2$ satisfy the following identities:
\begin{enumerate}
\item $\pi_1\Sigma(a,b)=a$ for all $a,b\in M$,
\item $\pi_2\Sigma(a,b)=b$ for all $a,b\in M$,
\item (Right Distibutivity) $\Sigma(a,b)c=\Sigma(ac,bc)$ for all $a,b,c\in M$, and
\item (Partition of Unity) $\Sigma(\pi_1,\pi_2)=1$.
\end{enumerate}
\end{proposition}

\begin{proof} If $(M,\Sigma,\pi_1,\pi_2)$ is a CP monoid, then (1) and (2) are satisfied by definition. Furthermore, $\Sigma$ is bijective and $\Sigma^{-1}=\Pi$ where $\Pi(m)=(\pi_1m,\pi_2m)$. So to prove identities (3) and (4), it suffices to check equality after applying $\Pi$ to both sides. For identity (3), we obtain $(ac,bc)$ on each side, while for identity (4) we obtain $(\pi_1,\pi_2)$ on each side, proving both identities.

Now, say that $M$, $\Sigma$, $\pi_1$ and $\pi_2$ satisfy all four identities. To see that $(M,\Sigma,\pi_1,\pi_2)$ is a CP monoid, we need only see that $\Sigma$ is bijective. Taking $\Pi$ as above, identities (1) and (2) show that $\Pi\Sigma(a,b)=(a,b)$ for all $a,b\in M$ so that $\Sigma$ is injective. But, if $m\in M$, then identities (3) and (4) show that $\Sigma(\pi_1m,\pi_2m)=\Sigma(\pi_1,\pi_2)m=m$, so that $\Sigma$ is also surjective.
\end{proof}

\begin{remark}\label{Rem:UA} From the point of view of universal algebra, a CP monoid is an algebra with three 0-ary operations ($1$, $\pi_1$ and $\pi_2$), two binary operations (the monoid operation and $\Sigma$) and satisfying seven identities (associativity, left and right identity and the four identities in Proposition \ref{Prop:CPalt}). In the main construction of this paper it is convenient to put aside partition of unity to impose later. So, we make the following auxiliary definition.
\end{remark}

\begin{definition}\label{Def:CQP} A \emph{categorical quasiproduct (CQP)} monoid $(M,\Sigma,\pi_1,\pi_2)$ is a monoid $M$, together with two elements $\pi_1,\pi_2\in M$ and a function $\Sigma:M^2\rightarrow M$ such that for all $a,b,c\in M$:
\begin{enumerate}
\item $\pi_1\Sigma(a,b)=a$
\item $\pi_2\Sigma(a,b)=b$
\item $\Sigma(ac,bc)=\Sigma(a,b)c$.
\end{enumerate}
\end{definition}

We often write \lq\lq$M$ is a CP (or CQP) monoid" when $\Sigma$, $\pi_1$ and $\pi_2$ are clear from context.

\begin{remark}\label{Rem:Pi} Note that in Definition \ref{Def:CP} we require $\Sigma$ to be a bijection, but do not require this in Definition \ref{Def:CQP}. However, the function $\Pi(m)=(\pi_1m,\pi_2m)$ from the proof of Proposition \ref{Prop:CPalt} is a left inverse for $\Sigma$. Thus, in a CQP monoid $\Sigma$ is always injective. In a CP monoid, $\Pi$ is the inverse for $\Sigma$. Although $\Sigma(\pi_1,\pi_2)$ is not $1$ in a CQP monoid that is not a CP monoid, the following lemma shows that it is very close to being a left identity.
\end{remark}

\begin{lemma}\label{Lem:Left1} Let $(M,\Sigma,\pi_1,\pi_2)$ be a CQP monoid and let $m\in M$ be in the image of $\Sigma$. Then $\Sigma(\pi_1,\pi_2)m=m$.
\end{lemma}

\begin{proof} Let $m=\Sigma(m_1,m_2)$. Then
$$\Sigma(\pi_1,\pi_2)m=\Sigma(\pi_1m,\pi_2m)=\Sigma(m_1,m_2)=m.$$
\end{proof}

\begin{remark}\label{Rem:Generalization} In Definition \ref{Def:dCP} we define a categorical $d$-fold product ($d$-CP) monoid, with $d$ distinguished elements and analogous $\Sigma:M^d\rightarrow M$. We could also define a $d$-CQP monoid, but do not have occasion to use one in this paper. Definitions \ref{Def:CP} and \ref{Def:CQP} are for the most interesting case of $d=2$. All of the arguments in this paper have straightforward generalizations to the case of any finite $d\ge 2$. For the sake of clarity, we just develop the case of $d=2$, leaving the generalizations to the interested reader.
\end{remark}

\begin{example}\label{Ex:VS} Let $V$ be a countable dimensional vector space with basis $\{e_0,e_1,e_2,\dots\}$ and let $M=\End(V)$. Let $\pi_1,\pi_2,\pi_1^*,\pi_2^*\in M$ be defined by
$$
\begin{array}{ll}
\pi_1(e_{2i})=0 & \pi_2(e_{2i})=e_i\\
\pi_1(e_{2i+1})=e_i & \pi_2(e_{2i+1})=0\\
\pi_1^*(e_i)=e_{2i+1} & \pi_2^*(e_i)=e_{2i}
\end{array}
$$
for $i\ge 0$. Then $\pi_j\pi_j^*=1$, $\pi_j\pi_k^*=0$ if $j\ne k$ and $\pi_1^*\pi_1+\pi_2^*\pi_2=1$. Define $\Sigma:M^2\rightarrow M$ by
$$\Sigma(f,g)=\pi_1^*f+\pi_2^*g.$$
All four of the identities of Proposition \ref{Prop:CPalt} are satisfied, so that $(M,\Sigma,\pi_1,\pi_2)$ is a CP monoid.
\end{example}

\begin{example}\label{Ex:FBS}
This example describes the monoids $F$, $B$ and $S$ that appear several times in this paper.

Let $F$ be the free monoid on symbols $\pi_1$ and $\pi_2$. Let $B$ be the monoid generated by symbols $\pi_1,\pi_2,\pi_1^*,\pi_2^*$ and $0$ subject only to the following relations:
\begin{enumerate}
\item $0b=b0=0$ for all $b\in B$ and
\item $\pi_i\pi_j^*=\delta_{ij}$, where $\delta_{ij}=0,1\in B$ is the Kronecker Delta.
\end{enumerate}

$F$ is a submonoid of $B$. If we further write $\pi_i^{**}=\pi_i$, $0^*=0$ and $(p_1\cdots p_d)^*=p_d^*\cdots p_1^*$, where each $p_i$ is $\pi_1$ or $\pi_2$, then $*$ is an antiautomorphism of $B$. If $w\in F$, then $ww^*=1$. Furthermore, every non-zero $b\in B$ can be uniquely written as $b=v^*w$ for $v,w\in F$.

Let $S=\{X\subseteq B|0\in X\}$. Then $S$ is a monoid under the multiplication law:
$$XY=\{xy|x\in X, y\in Y\}.$$
The identity in $S$ is $\{0,1\}$ and $S$ has an absorbing element $\{0\}$. There is an injective homomorphism $B\rightarrow S$ given by $b\mapsto\{0,b\}$ and we will often conflate $B$ with its image in $S$, writing $\{0,b\}$ as simply $b$. For $b\in B$ and $X\in S$, $bX=\{0,b\}X$ and $Xb=X\{0,b\}$, so this convention should cause no confusion.

It is easy to see that multiplication distributes on either side over union $\cup$: $X(Y\cup Z)=XY\cup XZ$ and $(Y\cup Z)X=YX\cup ZX$. Indeed, $S$ is a semiring with additive structure given by $\cup$.

This distributive property, together with $\pi_i\pi_j^*=\delta_{ij}$, gives that $S$ is a CQP monoid with $\Sigma:S^2\rightarrow S$ given by
$$\Sigma(X,Y)=\pi_1^*X\cup\pi_2^*Y.$$
However, we see from Proposition \ref{Prop:CPalt} that $S$ is not a CP monoid since 
$$\Sigma(\pi_1,\pi_2)=\{0,\pi_1^*\pi_1,\pi_2^*\pi_2\}\ne 1.$$
\end{example}

\begin{remark} In the construction of $S$ from $B$ in Example \ref{Ex:FBS}, we could have harmlessly added the restriction that the elements of $S$ be finite sets containing $0$, rather than arbitrary sets containing $0$. With or without the finiteness condition, the construction generalizes to a functor from the category of monoids with absorbing element to the category of semirings. We do not have occasion to apply this functor to other monoids in this paper.
\end{remark}

That $S$ in Example \ref{Ex:FBS} contains a free monoid generated by $\pi_1$ and $\pi_2$ is not a special feature. This is shared by all non-zero CQP monoids.

\begin{proposition}\label{Prop:Free} Let $(M,\Sigma,\pi_1,\pi_2)$ be a non-zero CQP monoid. Then $\pi_1$ and $\pi_2$ are right-invertible, but not left-invertible. Furthermore, $\pi_1$ and $\pi_2$ form a basis for a free submonoid of $M$.
\end{proposition}
\begin{proof} First note that the $\pi_1$ and $\pi_2$ are distinct. Otherwise for any elements $a,b\in M$, we have that
$$a=\pi_1\Sigma(a,b)=\pi_2\Sigma(a,b)=b$$
and $M$ has but one element, contradicting that $M$ is non-zero.

Now, let $\omega=\Sigma(1,1)$. Then $\pi_1\omega=\pi_2\omega =1$ so that $\pi_1$ and $\pi_2$ are \emph{simultaneously} right invertible. If one $\pi_i$ has a left inverse $\sigma$, then $\sigma =\sigma(\pi_i\omega)=(\sigma\pi_i)\omega =\omega$. Then for $j\ne i$ we have that $\pi_j=\pi_j(\omega\pi_i)=(\pi_j\omega)\pi_i=\pi_i$, a contradiction. So, neither $\pi_i$ has a left inverse.

Now, to see that $\pi_1$ and $\pi_2$ generate a free submonoid of $M$ we must see that the equation
$$\pi_{i_1}\cdots\pi_{i_s}=\pi_{j_1}\cdots\pi_{j_t}$$
implies that $s=t$ and $i_\ell =j_\ell$ for $1\le\ell\le s$. We may assume without loss of generality that $s\le t$. If $s=0$, so that the left hand side is $1$, then we must also have that $t=0$, since otherwise $\pi_{j_1}\cdots\pi_{j_{t-1}}$ is a left inverse of $\pi_{j_t}$, a contradiction. We now proceed by induction on $s$.

Let $0<s\le t$. Multiply both sides of the equation $\pi_{i_1}\cdots\pi_{i_s}=\pi_{j_1}\cdots\pi_{j_t}$ by $\omega=\Sigma(1,1)$ on the right to obtain $\pi_{i_1}\cdots\pi_{i_{s-1}}=\pi_{j_1}\cdots\pi_{j_{t-1}}$. Then by induction on $s$ we have that $t=s$ and $i_\ell=j_\ell$ for $\ell\le s-1$. It remains to see that $i_s=j_s$. If $i_s\ne j_s$ then without loss of generality, we may assume that $i_s=1$ and $j_s=2$.  Multiplying both sides of the original equation on the right by $\omega^\prime=\Sigma(1, \pi_2)$ we obtain that $\pi_{i_1}\cdots\pi_{i_{s-1}}=\pi_{j_1}\cdots\pi_{j_s}$ so that $s-1=s$ by induction, a contradiction if there ever was one! Thus $i_s=j_s$, completing the proof. 
\end{proof}

Recall that monoid homomorphisms must take $1$ to $1$ as well as respect multiplication. We define homomorphisms of CQP monoids in the canonical way:

\begin{definition}\label{Def:Hom} If $(M,\Sigma,\pi_1,\pi_2)$ and $(N,\Phi,\tau_1,\tau_2)$ are CQP monoids, then a CQP monoid homomorphism from $M$ to $N$ is a monoid homomorphism $f:M\rightarrow N$ such that $f(\pi_i)=\tau_i$ for $i=1,2$ and $f(\Sigma(a,b))=\Phi(f(a),f(b))$ for all $a,b\in M$. A CP monoid homomorphism is a CQP monoid homomorphism between CP monoids.
\end{definition}

\begin{remark}\label{Rem:Cat} The characterization of CP monoids via partition of unity in Proposition \ref{Prop:CPalt} shows that if $f:M\rightarrow N$ is a CQP monoid homomorphism and $M$ is a CP monoid, then $N$ is as well.

With this definition of homomorphism, CQP monoids form a category and CP monoids form a full subcategory. A 0 monoid is trivially a CP monoid and is a final object in both categories. These categories have initial objects $T$ and $U$ respectively that we construct in Section \ref{Section:Universal}. Remark \ref{Rem:Functor} describes a natural functor from CQP monoids to CP monoids. However, we have not investigated other constructions in these categories, such as limits and colimits.
\end{remark}

\begin{definition}\label{Def:Cong} If $(M,\Sigma,\pi_1,\pi_2)$ is a CQP monoid, then an equivalence relation $\sim$ on $M$ is a \emph{congruence} if whenever $a\sim b$ and $c\sim d$ we also have that $ac\sim bd$ and $\Sigma(a,c)\sim\Sigma(b,d)$.\end{definition}

\begin{remark}\label{Rem:Functor} Definition \ref{Def:Cong} is a specialization of the standard definition of a congruence relation in universal algebra, and we now review some basic facts as they apply to this paper. See \cite{G} for a more general treatment. Congruences $\sim$ are the equivalence relations whose equivalence classes $M/\!\!\sim$ inherit from $M$ the structure of a CQP monoid. For a congruence $\sim$, the function $M\rightarrow M/\!\!\sim$ taking an element to its equivalence class is a CQP monoid homomorphism and if $\phi:M\rightarrow N$ is any other CQP monoid homomorpism such that $\phi(a)=\phi(b)$ whenever $a\sim b$, then $\phi$ factors uniquely through $M\rightarrow M/\!\!\sim$.

Given any $a,b\in M$, there is a unique smallest congruence such that $a\sim b$, known as a principal congruence since it has one generator. In particular, if we let $\sim_{cp}$ be the principal congruence generated by $\Sigma(\pi_1,\pi_2)\sim_{cp}1$, then by Proposition \ref{Prop:CPalt} $M_{cp}:=M/\!\!\sim_{cp}$ is a CP monoid and by the paragraph above, the homomorphism $M\rightarrow M_{cp}$ is universal in the sense that for any CP monoid $N$ the induced map $\Hom(M_{cp},N)\rightarrow\Hom(M,N)$ is a bijection.

Furthermore, the construction of $M_{cp}$ respects CQP monoid homomorphisms, so that we obtain a functor {\bf cp} from the category of CQP monoids to the full subcategory of CP monoids that is the identity on CP monoids, together with a natural transformation from the identity functor to {\bf cp}.
\end{remark}

\section{Free Magmas and Trees}\label{Section:Free}
\begin{definition} A \emph{magma} $(Y,\Sigma)$ is a set $Y$ with a function $\Sigma:Y^2\rightarrow Y$. If $(Y,\Sigma)$ and $(Z,\Sigma^\prime)$ are magmas, then a function $f:Y\rightarrow Z$ such that $f(\Sigma(a,b))=\Sigma^\prime(f(a),f(b))$ for all $a,b\in Y$ is a \emph{magma homomorphism}. For a set $X$, the elements of the \emph{free magma} $Y_X$ with basis $X$ are partitioned by positive integers and defined as follows. The elements of $Y_X$ of degree $1$ are the elements of $X$. The elements of $Y_X$ of degree $d>1$ are ordered pairs $(A,B)$, where $A$ and $B$ are lower degree elements of $Y_X$ whose degrees sum to $d$. $\Sigma:Y_X^2\rightarrow Y_X$ is given by $\Sigma(A,B)=(A,B)\in Y_X$.
\end{definition}
 
 Thus, if $x,y,z\in X$ then $A=((y,z),x)$ and $B=(z,x)$ are two elements of $Y_X$ and $\Sigma(A,B)=(((y,z),x),(z,x))$. These definitions come from sections 1.1 and 7.1 of the first chapter of \cite{B}, except that Bourbaki writes \lq\lq length" where we write \lq\lq degree"; we reserve \lq\lq length" for elements of free monoids. The basic facts that we discuss in this section are obvious, but can also be found in section 7.1 of the first chapter of \cite {B}.
 
 However, in this paper, we view the elements of $Y_X$ in an equivalent way as the finite nonempty ordered binary trees with leaves colored by elements of $X$. Then, for such trees $A$ with root $r_A$ and $B$ with root $r_B$, $\Sigma(A,B)=C$ where $C$ is a new tree whose set of vertices is the disjoint union of the vertices of $A$, the vertices of $B$ and $\{r_C\}$. $r_C$ is the root of $C$ with left child $r_A$ and right child $r_B$. All other parent/child relationships and colorings are as in $A$ and in $B$. Taking $A$ and $B$ as in the previous paragraph, they and their composition are represented thus:

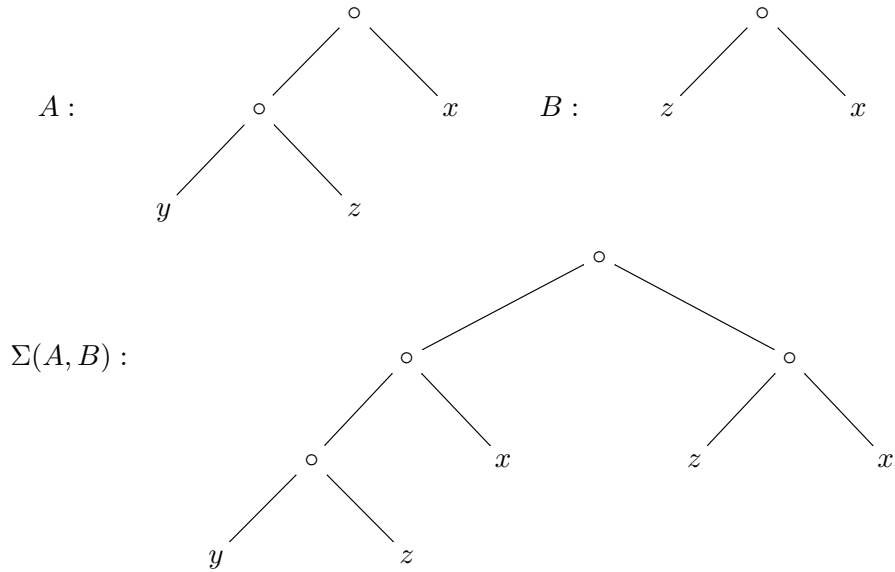
\[\xymatrix{
&&&\circ\ar@{-}[ld]\ar@{-}[rd]&&&&\circ\ar@{-}[ld]\ar@{-}[rd]&\\
A:&&\circ\ar@{-}[ld]\ar@{-}[rd]&&x&B:&z&&x\\
&y&&z&&&&&
}\]
\[\xymatrix{
&&&&&\circ\ar@{-}[lld]\ar@{-}[rrd]&&\\
\Sigma(A,B):&&&\circ\ar@{-}[ld]\ar@{-}[rd]&&&&\circ\ar@{-}[ld]\ar@{-}[rd]&\\
&&\circ\ar@{-}[ld]\ar@{-}[rd]&&x&&z&&x\\
&y&&z&&&&&
}\]

Then $\deg A$ becomes the number of leaves of the binary tree structure of $A$. We continue to identify the elements of $Y_X$ of degree $1$ with elements of $X$.

\begin{remark}\label{Rem:DegFactor} For $A\in Y_X$, either $A\in X$ or there are unique $A_1,A_2\in Y_X$ such that $A=\Sigma(A_1,A_2)$. It follows easily that for every function $\phi_0:X\rightarrow Z$, where $(Z,\Sigma^\prime)$ is a magma, there is a unique extension of $\phi_0$ to a magma homomorphism $\phi:Y_X\rightarrow Z$ inductively described for $A\in Y_X$ of degree greater than $1$ by $\phi(A)=\Sigma^\prime(\phi(A_1),\phi(A_2))$. This justifies the name \lq\lq free" magma.
\end{remark}

\section{Universal Solutions}\label{Section:Universal}
In this section we construct CQP monoids $T$ and $U$ such that $T$ is universal among CQP monoids, and $U$ is universal among CP monoids.

Here is a sketch of the constructions. Let $F$ be the free monoid on symbols $\pi_1$ and $\pi_2$ as in Example \ref{Ex:FBS}. Let $T=Y_F$ be the free magma with basis $F$ and let $\Sigma :T\times T\rightarrow T$ be the composition law in $T$ as in Section \ref{Section:Free}. We define a left action of $F$ on $T$ with the property that $\pi_1\Sigma(A,B)=A$ and $\pi_2\Sigma(A,B)=B$. Then, we extend this action to a binary operation on $T$ giving $T$ a monoid structure. With this monoid structure, $T$ is the initial object in the category of CQP monoids. Finally, we explicitly describe an equivalence relation $\equiv$ on $T$  and show that it is the same as the congruence $\sim_{cp}$ described in Remark \ref{Rem:Functor}. The quotient CQP monoid $U=T_{cp}$ is the initial object in the category of CP monoids.

For $F$ and $T$ as in the preceeding paragraph, we now define a left monoid action of $F$ on $T$ as a set. Since $F$ is free, an action by $F$ is uniquely determined by the actions of $\pi_1$ and $\pi_2$, which we may freely choose. 

\begin{definition}\label{Def:TCQP} Let $A\in T$. If $A\in F\subset T$, then $\pi_1$ and $\pi_2$ act on $A$ via left multiplication in $F$. Otherwise, by Remark \ref{Rem:DegFactor} there are unique $A_1,A_2\in T$ such that $A=\Sigma(A_1,A_2)$. In this case, $\pi_1 A=A_1$ and $\pi_2 A=A_2$.
\end{definition}

The action of a general $w\in F$ on $T$ is determined inductively, applying one symbol of $w$ at a time from right to left. This action is not hard to understand. For $A\in T$, the symbols of $w$ give instructions to move successively along the directed graph of $A$, starting at the root of $A$ and moving to a right child for each $\pi_2$ and to a left child for each $\pi_1$ until the symbols of $w$ are exhausted or a leaf is reached. If the symbols of $w$ are exhausted, then $wA$ is the tree of descendents of the vertex that has been reached. Otherwise, write $w=w_1w_2$ where $w_2$ consists of the string of symbols that have brought us to this leaf. If the leaf has color $v\in F$, then $wA=w_1v$.

\begin{example}\label{Ex:FxT}
Let $A\in T$ be as pictured below. On the following line are $\pi_2A$, $\pi_1^3A$ and $\pi_2\pi_1A$.
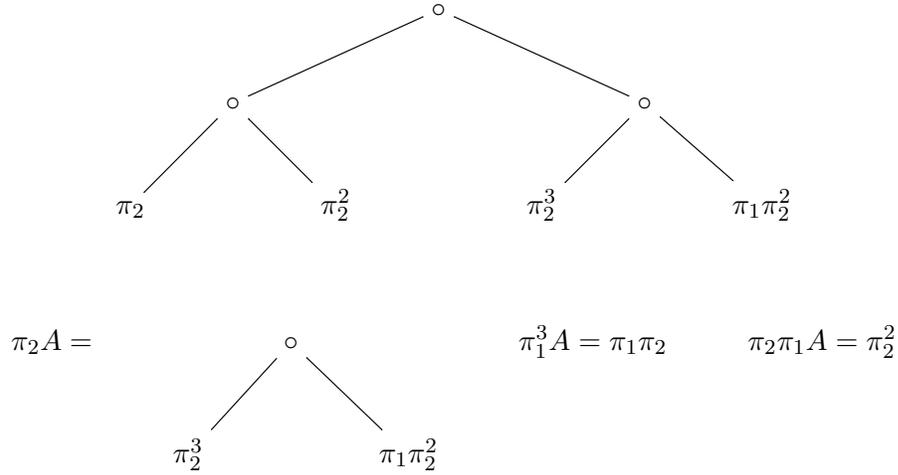
\[\xymatrix{
&&&\circ\ar@{-}[lld]\ar@{-}[rrd]&&&\\
&\circ\ar@{-}[ld]\ar@{-}[rd]&&&&\circ\ar@{-}[ld]\ar@{-}[rd]\\
\pi_2&&\pi_2^2&&\pi_2^3&&\pi_1\pi_2^2
}\]
\medskip
\[\xymatrix{
\pi_2A=&&\circ\ar@{-}[ld]\ar@{-}[rd]&&\pi_1^3A=\pi_1\pi_2&\pi_2\pi_1A=\pi_2^2\\
&\pi_2^3&&\pi_1\pi_2^2&&
}\]
\end{example}

We now extend the action $F\times T\rightarrow T$ to a multiplication $T\times T\rightarrow T$.

\begin{definition}\label{Def:Tmonoid} For $A,B\in T$ define $AB\in T$, the \emph{product} of $A$ and $B$, to be the tree obtained by replacing each leaf of $A$ of color $w\in F$ by $wB$.\end{definition}

\begin{example}\label{Ex:TxT}
Using the computations in Example \ref{Ex:FxT} we have that:
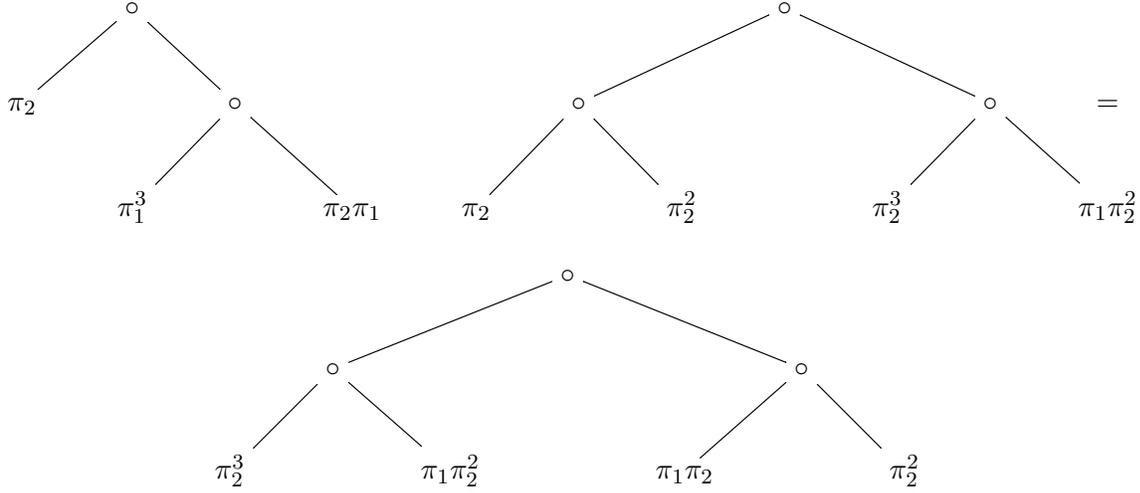
\[\xymatrix{
&\circ\ar@{-}[ld]\ar@{-}[rd]&&&&&&\circ\ar@{-}[lld]\ar@{-}[rrd]&&&\\
\pi_2&&\circ\ar@{-}[ld]\ar@{-}[rd]&&&\circ\ar@{-}[ld]\ar@{-}[rd]&&&&\circ\ar@{-}[ld]\ar@{-}[rd]&=\\
&\pi_1^3&&\pi_2\pi_1&\pi_2&&\pi_2^2&&\pi_2^3&&\pi_1\pi_2^2
}\]
\[\xymatrix{
&&&\circ\ar@{-}[lld]\ar@{-}[rrd]&&&\\
&\circ\ar@{-}[ld]\ar@{-}[rd]&&&&\circ\ar@{-}[ld]\ar@{-}[rd]\\
\pi_2^3&&\pi_1\pi_2^2&&\pi_1\pi_2&&\pi_2^2
}\]
\end{example}

\begin{remark}\label{Rem:RD} We will see in Theorem \ref{UniversalQ} that this product is associative, so that $T$ is a monoid with identity $1$. We will accomplish this by constructing an injective magma homomorphism from $T$ to the monoid $S$. In the mean time, note that $T$ satisfies the right distributive property $\Sigma(A,B)C=\Sigma(AC,BC)$. Indeed, $AC$ and $BC$ are computed by replacing each leaf of $A$ and of $B$ by a colored tree that depends only on $C$ and on the color of the leaf. Since $\Sigma(A,B)$ has the same leaves with the same colors as those of $A$ and $B$, we may make this replacement before or after applying $\Sigma$ with the same effect.
\end{remark}

\begin{proposition}\label{Prop:Thom} Let $F$ be the free monoid and $T=Y_F$ be the free magma defined above. Let $(N,\Sigma^\prime,\pi_1^\prime,\pi_2^\prime)$ be a CQP monoid, let $\phi_0:F\rightarrow N$ be the unique monoid homomorphism such that $\phi_0(\pi_1)=\pi_1^\prime$ and $\phi_0(\pi_2)=\pi_2^\prime$ and let $\phi:T\rightarrow N$ be the unique magma homomorphism from $(T,\Sigma)$ to $(N,\Sigma^\prime)$ extending $\phi_0$. Then for every $A,B\in T$, $\phi(AB)=\phi(A)\phi(B)$, where $AB$ is as in Definition \ref{Def:Tmonoid}.
\end{proposition}

\begin{proof}
We start by assuming that $A\in F$. We will proceed by induction on the length of $A$, with the case of $A=1$ being clear. So we may assume that $A=w\pi_i$ for $i=1$ or $2$ and that $\phi(wC)=\phi(w)\phi(C)$ for all $C\in T$ by induction. If $B\in F$, then we are already done, so assume that $B=\Sigma(B_1,B_2)$. Then
$$\phi(AB)=\phi(wB_i)=\phi(w)\phi(B_i)=\phi(w)\pi_i^\prime\Sigma^\prime(\phi(B_1),\phi(B_2))=\phi(w\pi_i)\phi(B)=\phi(A)\phi(B)$$
as required. Now, we proceed by induction on $\deg A$ and let $A=\Sigma(A_1,A_2)$. Then $\phi(A_iB)=\phi(A_i)\phi(B)$ by induction. So, using right distributivity in $T$ and in $N$ we have
$$\phi(AB)=\phi\left(\Sigma(A_1B,A_2B)\right)=\Sigma^\prime\left(\phi(A_1B),\phi(A_2B)\right)=\Sigma^\prime\left(\phi(A_1),\phi(A_2)\right)\phi(B)=\phi(A)\phi(B)$$
as required.
\end{proof}

We now consider an important special case.

\begin{definition}\label{Def:Beta} Let $(S,\Sigma,\pi_1,\pi_2)$ be the CQP monoid described in Example \ref{Ex:FBS}. The \emph{branch homomorphism} is the magma homomorphism $\beta:T\rightarrow S$ induced by the map $\beta_0:F\rightarrow S$ taking $w$ to $\{0,w\}$, as in Proposition \ref{Prop:Thom}.
\end{definition}

The name \lq\lq branch homomorphism" is justified by the following.

\begin{lemma}\label{Lem:Branch}
For $A\in T$, the non-zero elements $w_1^*w_2$ of $\beta(A)$ are in one-to-one correspondence with the leaves of $A$, where $w_1$ read from right to left describes the path from the root to that leaf and $w_2$ is the color of that leaf.
\end{lemma}

\begin{proof} We induct on $\deg A$. If $\deg A=1$ so that $A=w\in F$, then $\beta(A)=\{0,w\}$ and we are done. Otherwise, say that $A=\Sigma(A_1,A_2)$. Then 
$$\beta(A)=\Sigma_S(\beta(A_1),\beta(A_2))=\pi_1^*\beta(A_1)\cup\pi_2^*\beta(A_2).$$
Keeping in mind that $*$ is an antiautomorphism of the branch monoid $B$, as in Example \ref{Ex:FBS}, the claim follows immediately by induction.
\end{proof}

\begin{theorem}\label{UniversalQ} $T$ is a CQP monoid. Let $(N,\Sigma^\prime,\pi_1^\prime,\pi_2^\prime)$ be another CQP monoid. Then there is a unique CQP monoid homomorphism $\phi:T\rightarrow N$.
\end{theorem}

That is, $T$ is the initial object in the category of $CQP$ monoids.

\begin{proof} The monoid structure on $T$ will be the product described in Definition \ref{Def:Tmonoid}, which has $1\in F$ as an identity. Indeed, we have already established in Definition \ref{Def:TCQP} and Remark \ref{Rem:RD} all of the identities required by Definition \ref{Def:CQP}. It only remains to see that the product of Definition \ref{Def:Tmonoid} is associative. We will show this by showing that the branch homomorphism $\beta:T\rightarrow S$ is injective, so that $T$ inherits associativity from $S$ by way of Proposition \ref{Prop:Thom}.

First note that the number of non-zero elements of $\beta(A)$ is $\deg A$ by Lemma \ref{Lem:Branch}. Now, consider $A,B\in T$ such that $\beta(A)=\beta(B)$. Then $\deg A=\deg B$. If this degree is $1$, then clearly $A=B$. Otherwise, $A=\Sigma(A_1,A_2)$ and $B=\Sigma(B_1,B_2)$. Then
$$\pi_1^*\beta(A_1)\cup\pi_2^*\beta(A_2)=\pi_1^*\beta(B_1)\cup\pi_2^*\beta(B_2).$$
Comparing the elements on each side whose left-most factor is $\pi_1^*$ or $\pi_2^*$, we see for $i=1,2$ that $\pi_i^*\beta(A_i)=\pi_i^*\beta(B_i)$. Since left multiplication by $\pi_i^*$ in the branch monoid of Example \ref{Ex:FBS} is injective, $\beta(A_i)=\beta(B_i)$ and we have by induction that $A_i=B_i$. Thus, $A=B$ and $\beta$ is injective.

It follows that the magma homomorphism $\phi:T\rightarrow N$ of Proposition \ref{Prop:Thom} is a CQP monoid homomorphism and is the unique CQP homomorphism from $T$ to $N$.
\end{proof}

Having now seen that $T$ is a CQP monoid, we also see that the branch homomorphism $\beta$ is the unique CQP monoid homomorphism $\beta:T\rightarrow S$ guaranteed by the theorem.

\begin{definition}\label{Def:U} $U$ is the CP monoid $T_{cp}=T/\!\!\sim_{cp}$, where $\sim_{cp}$ is as described in Remark \ref{Rem:Functor}.\end{definition}

It follows immediately from Theorem \ref{UniversalQ} and Remark \ref{Rem:Functor} that $U$ is an initial object in the category of CP monoids. To find out more about $U$, we need a more explicit description of $U$, which amounts to a more explicit description of the congruence $\sim_{cp}$ on $T$ generated by the relation $\Sigma(\pi_1,\pi_2)\sim_{cp} 1$. 

By the right distributive property, we have for every $w\in F$ that 
$$w\sim_{cp}\Sigma(\pi_1,\pi_2)w=\Sigma(\pi_1w,\pi_2w),$$
the tree with just two leaves which have colors $\pi_1w$ and $\pi_2w$ in order. Applying $\Sigma$ repeatedly, we see that whenever $A$ is a tree with a leaf of color $w$, we must be able to replace that leaf by $\Sigma(\pi_1 w, \pi_2 w)$ and remain in the congruence class. We will now examine the equivalence relation defined in that way. We will show that the equivalence relation is a congruence, so that it agrees with $\sim_{cp}$.

\begin{definition}\label{Def:Congruence} Let $\equiv$ be the equivalence relation on $T$ generated by the relation $A\sim B$ if $B$ is obtained from $A$ by replacing a leaf of $A$ of color $w$ by $\Sigma(\pi_1 w, \pi_2 w)$. That is, $A\equiv B$ if one can change $A$ to $B$ through a finite sequence of moves of the form:
\begin{enumerate} 
\item\emph{Expansion:} a leaf of $A$ colored by $w$ is replaced by a degree $2$ tree $\Sigma(\pi_1 w,\pi_2 w)$, or
\item\emph{Retraction:}  a pair of leaves of $A$ having a common parent and colors $\pi_1 w$ and $\pi_2 w$ in order are retracted to a single leaf of color $w$.
\end{enumerate}
An element $A\in T$ is \emph{reduced} with respect to $\equiv$ if $A$ has no pair of leaves having a common parent and colors $\pi_1 w$ and $\pi_2 w$ with $w \in F$ in order.
\end{definition}

The following lemma shows that the equivalence classes of $\equiv$ in $T$ are in bijection with the reduced elements of $T$. Furthermore, the proof shows that given $A\in T$ it is straightforward to compute the reduced element of $T$ that $A$ is equivalent to. If $A$ is not reduced already, then apply a retraction move to an eligible pair of leaves of $A$. Continue doing so until there are no more eligible pairs of leaves.

\begin{lemma} Every $A\in T$ is equivalent to a unique reduced element of $T$ under $\equiv$, denoted $r(A)$. Furthermore, $r(A)$ may be obtained from $A$ through a chain of retraction moves alone.
\end{lemma}
\begin{proof} We first show that if $A\in T$, then there is a unique reduced $C\in T$ that can be obtained from $A$ through just a chain of retraction moves. We show this by induction on the degree of $A$. If $\deg A=1$, then $A$ is already reduced and no retraction move is possible. Now, assume that $\deg A>1$. If $A$ is reduced, then we are done. If $A$ is not reduced, then consider the pairs of leaves that may be retracted. If there is only one such pair, then we must begin by retracting that pair. But then the degree of $A$ has decreased and we can apply induction. If there are more such pairs than one, then consider any two distinct pairs. Say that the retraction of $A$ using the first pair brings us to $A_1$ and the retraction of $A$ using the second pair brings us to $A_2$. Then by induction there are unique reduced $C_1$ and $C_2$ that $A_1$ and $A_2$ can be reduced to through a sequence of retractions. Now, the pair of retractable leaves giving $A_2$ still exists and is retractable in $A_1$ and the pair of retractable leaves giving $A_1$ still exists and is retractable in $A_2$. Retracting this pair of leaves in each case brings us to the same $A_3$, in which both pairs of leaves have been retracted, and there is a reduced $C_3$ that $A_3$ can be transformed to through a sequence of retraction moves. By the uniqueness properties of $C_1$ and $C_2$, we obtain that $C_1=C_3=C_2$, proving the claim.

For $A\in T$ let $r(A)$ be the unique reduced element of $T$ obtainable from $A$ through a chain of retraction moves. We now show that if $A\equiv B$ then $r(A)=r(B)$. We prove this by induction on the smallest number of expansion and retraction moves required to obtain $B$ from $A$, which we denote by $d(A,B)$.

If $d(A,B)=0$, then $A=B$ and there is nothing to show. If $d(A,B)=1$, then either $B$ can be obtained from $A$ by a retraction move, or $A$ can be obtained from $B$ by a retraction move. Either way, $r(A)=r(B)$. If $d=d(A,B)>1$, then given a sequence of $d$ moves taking $A$ to $B$, let $C$ be the result of the first $d-1$ such moves. Then $r(A)=r(C)=r(B)$ by induction.

Now we prove the statement of the lemma. If $A\in T$, then $r(A)$ is a reduced element of $T$ that is equivalent to $A$ by construction. If $B\in T$ is a reduced element of $T$ such that $A\equiv B$, then we have that $B=r(B)=r(A)$ as required.
\end{proof}

\begin{proposition}\label{Prop:Congruence} $\equiv$ and $\sim_{cp}$ are the same relation on $T$.\end{proposition}
\begin{proof} We have already seen in the discussion preceeding Definition \ref{Def:Congruence} that $\sim_{cp}$ contains $\equiv$. Since $\sim_{cp}$ is the smallest congruence for which $\Sigma(\pi_1,\pi_2)\sim_{cp}1$, and we have that $\Sigma(\pi_1,\pi_2)\equiv 1$, we only need to see that $\equiv$ is a congruence. We must show for $A, B, C\in T$ such that $A\equiv B$, we have that
\begin{enumerate}
\item $\Sigma(A,C)\equiv\Sigma(B,C)$,
\item $\Sigma(C,A)\equiv\Sigma(C,B)$,
\item $AC\equiv BC$, and
\item $CA\equiv CB$.
\end{enumerate}
In each of these congruences, it suffices to consider the situation in which $B$ is obtained from $A$ by a single expansion move.

Then, cases (1) and (2) are obvious. The expansion of a leaf of $A$ to two leaves of $B$ may be accomplished before or after joining with $C$ giving the same result.

For case (3), we begin by considering the case in which $A=w\in F$, so that $B=\Sigma(\pi_1 w,\pi_2 w)$. If $wC=v\in F$, then
$$AC=v\equiv\Sigma(\pi_1 v,\pi_2 v)=\Sigma(\pi_1 wC,\pi_2 wC)=\Sigma(\pi_1 w,\pi_2 w)C=BC$$
as required. But if $wC$ has degree greater than $1$, then $wC$ is in the image of $\Sigma$. Then by Lemma \ref{Lem:Left1} we have that $BC=\Sigma(\pi_1,\pi_2)wC=wC=AC$ as required.

Now, for case (3) in general, we proceed by induction on $\deg A$ with the case of $\deg A=1$ having been settled above already. If $\deg A>1$, then let $A=\Sigma(A_1,A_2)$. The leaf expanded to form $B$ is in either $A_1$ or in $A_2$. If it is in $A_1$, then $B=\Sigma(B_1,A_2)$ where $B_1$ is obtained from $A_1$ through this expansion move. Then using induction and case (1), we obtain
$$AC=\Sigma(A_1C,A_2C)\equiv\Sigma(B_1C,A_2C)=\Sigma(B_1,A_2)C=BC.$$
A similar argument using induction and case (2) may be used if the leaf expanded to form $B$ is in $A_2$.

It remains to prove case (4), that $CA\equiv CB$. We first consider the case in which $C=w\in F$. If $\deg wA>1$, then either the leaf of $A$ that is extended in $B$ is also a leaf of $wA$ or it is not. If it is, then we only need to extend the same leaf to obtain $wB$ so that $wA\equiv wB$ as required. If it is not, then $wA=wB$, and the case of $\deg wA>1$ is resolved.

If $\deg wA=1$, then factor $w$ as $w=w_1w_2$ such that $w_2$ is the shortest right factor of $w$ with the property that $w_2A$ has degree $1$. Thus, if the leaf of $A$ that $w_2$ selects has color $v$, then $wA=w_1v\in F$. If this leaf is not the leaf that was extended in $B$, then $wA=wB$ and we are done. If it is, then $wB=w_1\Sigma(\pi_1 v,\pi_2 v)$. We consider two possibilities:
\begin{enumerate}
\item If $w_1=1$, then
$$wB=\Sigma(\pi_1 v,\pi_2 v)\equiv v=wA.$$
\item If $w_1=w_1^\prime\pi_i$ for $i=1$ or $2$, then
$$wB=w_1^\prime\pi_i\Sigma(\pi_1 v,\pi_2 v)=w_1^\prime(\pi_i v)=w_1v=wA.$$
\end{enumerate}
The case of $C=w$ is now resolved.

Now, if $\deg C>1$, then we may write $C=\Sigma(C_1,C_2)$ for $C_1$ and $C_2$ of lower degree. Then applying cases (1) and (2) with induction on $\deg C$ we obtain
$$CA=\Sigma(C_1A,C_2A)\equiv\Sigma(C_1B,C_2B)=CB$$
proving case (4). 
\end{proof}

Since each congruence class for $\equiv$ in $T$ contains a unique reduced element, we conflate the elements of $U$ with their reduced representatives in $T$. We carry out the monoid and $\Sigma$ operations in $U$ by combining these reduced representatives in $T$ and then, when necessary, reducing the result. In this way, we consider $F$ to be a submonoid of $U$. 

\begin{example}\label{Ex:UxU}
Reconsider the example of a product of two elements of $T$ in Example \ref{Ex:TxT}. Considering them instead in $U$, we have that the righthand pair of colors in the product (in $T$) is $\pi_1\pi_2$ and $\pi_2^2$,  which is of the form $\pi_1w,\pi_2w$ (in order). So we contract the pair of leaves with these colors leaving a leaf of color $w=\pi_2$:
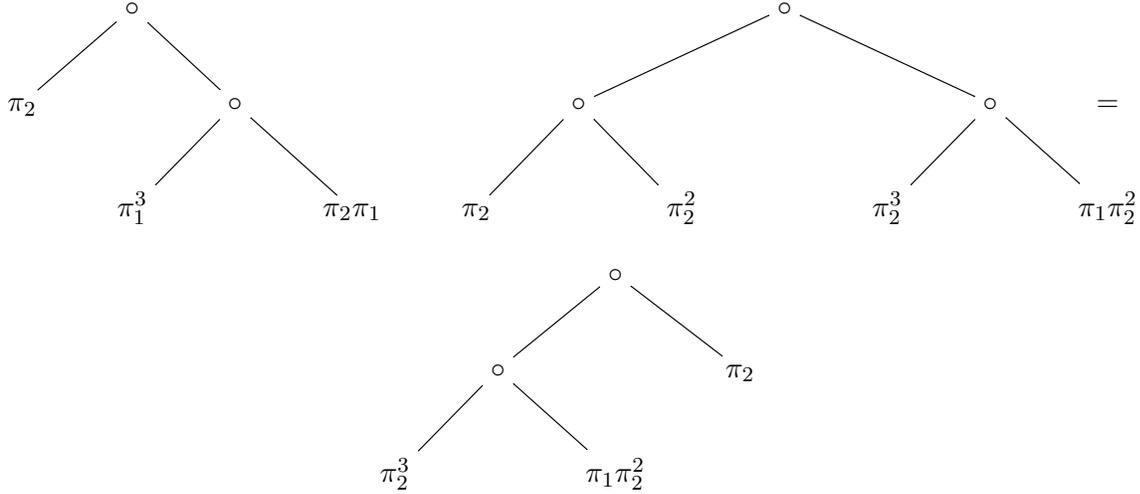
\[\xymatrix{
&\circ\ar@{-}[ld]\ar@{-}[rd]&&&&&&\circ\ar@{-}[lld]\ar@{-}[rrd]&&&\\
\pi_2&&\circ\ar@{-}[ld]\ar@{-}[rd]&&&\circ\ar@{-}[ld]\ar@{-}[rd]&&&&\circ\ar@{-}[ld]\ar@{-}[rd]&=\\
&\pi_1^3&&\pi_2\pi_1&\pi_2&&\pi_2^2&&\pi_2^3&&\pi_1\pi_2^2
}\]
\[\xymatrix{
&&\circ\ar@{-}[ld]\ar@{-}[rd]&\\
&\circ\ar@{-}[ld]\ar@{-}[rd]&&\pi_2\\
\pi_2^3&&\pi_1\pi_2^2&
}\]
Notice that the two leaves of colors $\pi_2^3,\pi_1\pi_2^2$ are not contracted, since the colors $\pi_2w,\pi_1w$ are in the wrong order.
\end{example}

If $A\in U$, then by $\deg A$ we mean the degree of this reduced representative in $T$. Since $\deg$ is additive for $\Sigma_T$, we have for $A,B\in U$ that $\deg\Sigma_U(A,B)\le\deg A+\deg B$. In fact, we have equality except for the case in which $A=\pi_1w$ and $B=\pi_2w$ when $\Sigma_U(A,B)=w$ has degree $1$.

\begin{theorem}\label{thm:Universal} $(U,\Sigma,\pi_1,\pi_2)$ is a CP monoid and is the initial object in the category of CP monoids. 
If $N$ is a non-zero CP monoid then the unique CP monoid homomorphism $\phi:U\rightarrow N$ is an injection.
\end{theorem}
\begin{proof} That $U$ is a CP monoid and the initial object in its category follow from Definition \ref{Def:U}, Theorem \ref{UniversalQ} and Remark \ref{Rem:Functor}. 

It only remains to see that if $N\ne 0$, then $\phi$ is injective. We will show by induction on $\deg A+\deg B$ that if $\phi(A)=\phi(B)$ then $A=B$. If $\deg A+\deg B=2$ then $A,B\in F$ and $A=B$, since $A\ne B$ would contradict that $\pi_1^\prime$ and $\pi_2^\prime$ form a basis for a free submonoid of $N$, as guaranteed by Proposition \ref{Prop:Free}. If $\deg A+\deg B>2$, then one of $A$ and $B$, say $B$, has degree greater than $1$. Then $B=\Sigma(B_1,B_2)$ and $\deg B_i<\deg B$ for each $i$. So, for $i=1,2$ we have that
$$\phi(\pi_iA)=\pi_i^\prime\phi(A)=\pi_i^\prime\phi(B)=\pi_i^\prime\phi(\Sigma(B_1,B_2))=\pi_i^\prime\Sigma^\prime(\phi(B_1),\phi(B_2))=\phi(B_i).$$
But, $\deg\pi_iA\le\deg A$, so that $\deg\pi_iA+\deg B_i<\deg A+\deg B$ and we may apply induction to see that $\pi_iA=B_i$ for $i=1,2$. Thus, $A=\Sigma(B_1,B_2)=B$ as required.
\end{proof}

\section{Invertible Elements}\label{Section:Invertible}
In this section, we characterize the left and right invertible elements of $T$ and of $U$. Then, for an arbitrary CP monoid $M$, we establish a bijection between units of $M$ and CP structures on $M$. 

\begin{proposition} $A\in T$ has a right inverse if and only if $\deg A=1$. $B\in T$ has a left inverse if and only if some leaf of $B$ has color $1$. \end{proposition}

\begin{proof} First note that $\deg(AB)\ge\deg(A)$ since $AB$ is formed by replacing each leaf of $A$ by a tree with one or more leaves itself. Since $\deg(1)=1$, for $A$ to have a right inverse it must have degree $1$. Say that $A=w\in F$ has length $d$. Then it is easy to see by induction on $d$ that $w(\Sigma(1,1))^d=1$.

Now, if $B\in T$ has a left inverse, then that left inverse must be of the form $w\in F$. Then by the construction of $wB$, if $wB=1$ then $w$ must describe a path exactly from the root of $B$ to a leaf of $B$ and that leaf has color $1$. Conversely, if $B$ has a leaf of color $1$, choose $w$ to trace the path from the root to that leaf. Then $wB=1$, proving the second statement.
\end{proof}

To study left and right inverses in $U$, we must be able to recognize when $A\in T$ represents $1\in U$. A criterion using the branch homomorphism $\beta:T\rightarrow S$ (Definition \ref{Def:Beta}) is given in the following lemma.

\begin{lemma}\label{Lemma:Recognition} If $A\in T$ and $w\in F\subseteq T$, then $A\equiv w$ if and only if every non-zero element of $\beta(A)\in S$ is of the form $v^*vw$ for some $v\in F$.
\end{lemma}
\begin{proof} Say that $\deg A=1$ so that $A=x\in F$. Then $\beta(A)=\{0,x\}$, the only possible value of $v$ is $1$ and the statement is clear. Now we proceed by induction on $\deg A$.

Say that $\deg A>1$ and $A\equiv w\in F$. Let $A_i=\pi_i A$ for $i=1,2$ so that $\deg A_i<\deg A$ and $A_i\equiv \pi_iw$. By induction, each non-zero element of $\beta(A_i)$ is of the form $x^*x\pi_iw$ for some $x\in F$. But $\beta(A)=\pi_1^*\beta(A_1)\cup\pi_2^*\beta(A_2)$ so that each non-zero element of $\beta(A)$ is of the form $\pi_i^*x^*x\pi_iw=v^*vw$ where $v=x\pi_i$ as required.

If $\deg A>1$, then every non-zero element of $\beta(A)$ is of the form $x^*y$ with $x,y\in F$ and $x\ne 1$. If there is a $w\in F$ such that every non-zero element of $\beta(A)$ is of the form $v^*vw$, then we partition these elements into those such that the right-most symbol in $v$ is $\pi_1$ or $\pi_2$, since $v$ is never $1$. Then the left-most symbol of $v^*$ is correspondingly $\pi_1^*$ or $\pi_2^*$. Write each such $v$ as $v=x\pi_i$. Taking $A_i=\pi_iA$ as above, then $\beta(A)=\pi_1^*\beta(A_1)\cup\pi_2^*\beta(A_2)$, so that each element of $\beta(A_i)$ is of the form $x^*x\pi_iw$ for $x\in F$. By induction, $A_i\equiv\pi_iw$ so that
$$A=\Sigma(A_1,A_2)\equiv\Sigma(\pi_1w,\pi_2w)\equiv w.$$
\end{proof}

We will use the following basic notions to characterize left and right invertible elements of $U$.

\begin{definition}\label{Def:CoDepInd} Let $G$ be a free monoid and $w_1,\dots ,w_n$ a family of elements of $G$. This family is said to be
\begin{enumerate}
\item \emph{left cofinite} if all but finitely many $y\in G$ can be written $y=xw_i$ for some $x$ and $i$;
\item \emph{left dependent} if $w_i=xw_j$ for some $x\in G$ and $i\ne j$;
\item \emph{left independent} if not left dependent.
\end{enumerate}
A left cofinite family is \emph{minimally left cofinite} if it does not remain left cofinite after any member of the family is removed. A left independent family $w_1,\dots ,w_n$ is \emph{maximally left independent} if for every $w_{n+1}\in G$, $w_1,\dots ,w_{n+1}$ is left dependent.
\end{definition}

Note that if $w_i=w_j$ for some pair of distinct $i$ and $j$, then the family is left dependent, but repeats in the family can be ignored when considering left cofiniteness.
 
In our characterizations of left and right invertibility of $A\in U$, the monoid of Definition~\ref{Def:CoDepInd} is the free monoid $F$ and the family of elements of $F$ is the family of colors of leaves of a representative $\tilde{A}\in T$ of $A$. We first see that these properties do not depend on which representative we choose.

\begin{lemma}\label{Lemma:Choice} If $A, B\in T$ and $A\equiv B$, then the family of colors of leaves of $A$ and the family of colors of leaves of $B$ are
\begin{enumerate}
\item both left cofinite or both not left cofinite and are
\item both left dependent or both left independent.
\end{enumerate}
\end{lemma}

\begin{proof} It suffices to consider the case in which $B$ is obtained from $A$ by a single expansion move.

For left cofiniteness, say that the leaf of $A$ that has been expanded has color $w\in F$. Note that if $z=vw$ where $v\ne 1$, then $z$ is a left multiple of exactly one of $\pi_1w$ and $\pi_2w$. Thus, if $y\in F$ and $y\ne w$, then $y$ is a left multiple of a color of a leaf of $A$ if and only if it is a left multiple of a color of a leaf of $B$. So, the set of elements of $F$ that are not left multiples of colors of $A$ differs from the set of elements of $F$ that are not left multiples of colors of $B$ by at most the single element $w$. Thus, one set is finite if and only if the other is, proving $(1)$.

For left dependency, up to reindexing we may say that the colors of the leaves of $A$ are $w_1,\dots ,w_n$ and the colors of the leaves of $B$ are $w_1,\dots ,w_{n-1},\pi_1w_n,\pi_2w_n$.

Say that $w_i=xw_j$ for some $i\ne j$. If $i,j<n$, then the colors of the leaves of $A$ and $B$ are both left dependent. If $i=n$ then this equation shows a left dependency for $A$ and $(\pi_1w_n)=\pi_1xw_j$ shows a left dependency for $B$. If $j=n$, then this equation shows a left dependency for $A$. If further $x=1$, then $\pi_1w_i=(\pi_1w_n)$ shows a left dependency for $B$. If $x\ne 1$, then $x=y\pi_k$ for $k=1$ or $2$ and $w_i=y(\pi_kw_n)$ shows a left dependency for $B$. This shows that if the family of colors of the leaves of $A$ is left dependent, then this is the case for $B$ as well.

Now, if the family of colors of the leaves of $B$ is left dependent, then since $\pi_1w_n$ and $\pi_2w_n$ are not right divisors of one another, the only ways that a dependency can occur are treated in the paragraph above. Thus, if the family of colors of the leaves of $B$ is left dependent, then this is the case for $A$ as well.
\end{proof}

\begin{theorem}\label{Thm:Left} Let $B\in U$ be represented by $\tilde{B}\in T$. Then $B$ has a left inverse in $U$ if and only if the family of colors of leaves of $\tilde{B}$ is left cofinite.
\end{theorem}

\begin{proof} Say that there is an $A\in U$ such that $AB=1$.  Let $d$ be the length of the longest path from the root of $\tilde{B}$ to a leaf. Let $\tilde{A}$ be a representative for $A$ in $T$. After applying suitably many expansion moves to $\tilde{A}$, we may assume that ever leaf of $\tilde{A}$ has a color of length greater than or equal to $d$ and that the tree structure of $\tilde{A}$ is a complete binary tree of depth $e$. (The unique binary tree all of whose leaves are a distance $e$ from the root.) Then the tree structure of $\tilde{A}\tilde{B}$ is also a complete binary tree of depth $e$. Since $\tilde{A}\tilde{B}\equiv 1$, Lemma \ref{Lemma:Recognition} tells us that every non-zero element of $\beta(\tilde{A}\tilde{B})$ is of the form $v^*v$ for $v\in F$, and there is exactly one such element for each $v\in F$ of length $e$ since the tree structure of $\tilde{A}\tilde{B}$ is a complete binary tree of depth $e$. By the construction of $\tilde{A}\tilde{B}$, each such $v$ is a left multiple of a color of a leaf of $\tilde{B}$. So any element of $F$ that is not a left multiple of a color of a leaf of $\tilde{B}$ has length less than $e$. There are only finitely many such elements of $F$, so that the family of colors of leaves of $\tilde{B}$ is left cofinite as required.

Now, say that the family of colors of leaves of $\tilde{B}$ is left cofinite. Then there is a $d$ such that every $v\in F$ of length $d$ is a left multiple of a color of a leaf of $\tilde{B}$. We must find $\tilde{A}$ such that $\tilde{A}\tilde{B}\equiv 1$. Choose $\tilde{A}$ to have a tree structure that is a complete binary tree of degree $d$. Then the non-zero elements of $\beta(\tilde{A})$ are of the form $v^*w_v$ where $v$ varies over all elements of $F$ of length $d$ and the colors $w_v$ have yet to be chosen. For each such $v$, there is a leaf of $\tilde{B}$ with a color $x$ such that $v=yx$. Let $z$ be the element of $F$ that selects this leaf from $\tilde{B}$ so that $z\tilde{B}=x$. Let $w_v=yz$. With this choice of colors for the leaves of $\tilde{A}$ we have that the non-zero elements of $\beta(\tilde{A})$ are precisely the $v^*v$ where $v$ has length $d$. Then Lemma \ref{Lemma:Recognition} shows that $\tilde{A}\tilde{B}\equiv 1$ and if $A$ is the image of $\tilde{A}$ in $U$ then $AB=1$ as required.
\end{proof}

\begin{theorem}\label{Thm:Right} Let $A\in U$ be represented by $\tilde{A}\in T$. Then $A$ has a right inverse in $U$ if and only if the family of colors of the leaves of $\tilde{A}$ is left independent.
\end{theorem}
\begin{proof}
Say that the family of colors of the leaves of $\tilde{A}$ is left independent. Thanks to Lemma \ref{Lemma:Choice} we harmlessly apply some expansion moves to $\tilde{A}$ so that we may assume that all of the colors of the leaves of $\tilde{A}$ have the same length $d$. Since the leaves of $\tilde{A}$ have left independent colors they are all distinct. Choose $\tilde{B}$ to have a tree structure that is a complete binary tree of depth $d$. Each color of $\tilde{A}$ selects a distinct leaf of $\tilde{B}$. Thus, the non-zero elements of $\beta(\tilde{A}\tilde{B})$ are of the form $v^*w_v$ where $v$ varies over elements of $F$ representing the paths from the roots of $\tilde{A}$ to its leaves, and the $w_v$ are colors yet to be chosen of distinct leaves of $\tilde{B}$. For these leaves of $\tilde{B}$, we choose $w_v=v$, and assign any remaining colors arbitrarily. Then $\tilde{A}\tilde{B}\equiv 1$ by Lemma~\ref{Lemma:Recognition} and the image of $\tilde{B}$ in $U$ is a right inverse for $A$.

Now, say that $A\in U$ has a right inverse $B\in U$ and choose a representative $\tilde{B}\in T$. After applying some expansion moves to $\tilde{A}$ and to $\tilde{B}$, we may assume that the colors of all of the leaves of $\tilde{A}$ have the same length $d$ and that $\tilde{B}$ is a complete binary tree of depth $d$. Thus, in the product $\tilde{A}\tilde{B}$ the colors of the leaves of $\tilde{A}$ are replaced by selected colors of the leaves of $\tilde{B}$. Since $AB=1\in U$, the non-zero elements of $\beta(\tilde{A}\tilde{B})$ are all of the form $v^*v$, where $v$ varies over elements of $F$ reperesenting the paths from the roots of $\tilde{A}$ to its leaves, and the color of the leaf of $\tilde{A}$ chosen by $v$ itself chooses a leaf of $\tilde{B}$ that has color $v$. It follows that the colors of the leaves of $\tilde{A}$ are all distinct. Since they are all of the same length, none is a left-multiple of another and so they form a left independent family as required.
\end{proof}

It follows from Theorems \ref{Thm:Left} and \ref{Thm:Right} that the units in $U$ are those trees whose colors are both left cofinite and left independent. We give the following alternative characterizations of the pair of these conditions together.

\begin{lemma}\label{Lem:Units} Let $G$ be a non-zero free monoid and let $w_1,\dots ,w_n\in G$. The following conditions for this family are equivalent:
\begin{enumerate}
\item The family is both left cofinite and left independent.
\item The family is minimally left cofinite.
\item The family is maximally left independent.
\end{enumerate}
\end{lemma}

\begin{proof} Say that $w_1,\dots ,w_n$ form a left cofinite and left independent family. Since $G$ is infinite, $n>0$ by cofiniteness. Choose $1\le i\le n$. Then by independence, $w_i$ is not a left multiple of any $w_j$ with $j\ne i$. Then for every $v\in G$, $vw_i$ is not a left multiple of $w_j$ with $j\ne i$ since $G$ is a free monoid. Since there are infinitely many such $vw_i$, leaving out $w_i$ would make the set not left cofinite. Therefore the family is minimally left cofinite.

Now, say that $w_1,\dots ,w_n$ form a left cofinite and left \emph{dependent} family. Then there are distinct $i$ and $j$ such that $w_i=vw_j$. Then $w_1,\dots ,w_{i-1},w_{i+1},\dots ,w_n$ is also left cofinite. Therefore, if $w_1,\dots ,w_n$ is minimally left cofinite, then it is also left independent. We have shown that conditions (1) and (2) are equivalent.

Now, say again that $w_1,\dots ,w_n$ form both a left cofinite and left independent family. If the family is not maximally left independent, then we may extend it to another left independent family $w_1,\dots ,w_{n+1}$. But, this family is also left cofinite since $w_1,\dots ,w_n$ is. In particular, it is not minimally left cofinite, but this contradicts the equivalence of cases (1) and (2). Thus, (1) implies (3).

Finally, say that $w_1,\dots ,w_n$ is maximally left independent. Then for every $x\in G$, $w_1,\dots ,w_n,x$ is left dependent. Thus, either $x$ is a left multiple of some $w_i$ or $x$ is a right divisor of some $w_i$. But, there are only finitely many right divisors of $w_1,\dots ,w_n$, so the family is left cofinite. Thus, (3) implies (1), completing the proof.
\end{proof}

\begin{corollary}\label{Cor:Units} Let $A\in U$ be represented by $\tilde{A}\in T$. Then $A$ is a unit in $U$ if and only if the family of colors of leaves of $\tilde{A}$ satisfies the equivalent conditions of Lemma \ref{Lem:Units}.\end{corollary}

We close this section by considering the relationship between left and right invertible elements of a CP monoid $M$ and alternative CQP or CP monoid structures on $M$.

\begin{theorem}\label{Thm:CPclass} Let $(M,\Sigma,\pi_1,\pi_2)$ be a CP monoid. Then there is a bijection
$$\{(f,g)\in M^2|fg=1\}\leftrightarrow\{(\Phi,\tau_1,\tau_2)|(M,\Phi,\tau_1,\tau_2)\text{ is a CQP monoid.}\}$$
that restricts to a bijection
$$\{(u,u^{-1})\in M^2|u\in M^*\}\leftrightarrow\{(\Phi,\tau_1,\tau_2)|(M,\Phi,\tau_1,\tau_2)\text{ is a CP monoid.}\}.$$
\end{theorem}

\begin{proof} Take a pair $(f,g)$ such that $fg=1$ to the triple $(g\Sigma,\pi_1f,\pi_2f)$. It is easy to see that $(M,g\Sigma,\pi_1f,\pi_2f)$ is a CQP monoid. Indeed,
$(\pi_1f)(g\Sigma(a,b))=\pi_1\Sigma(a,b)=a$, and similarly for $\pi_2f$. Right distributivity for $g\Sigma$ follows from right distributivity for $\Sigma$. If furthermore $f=u$ and $g=u^{-1}$, then $u^{-1}\Sigma(\pi_1u,\pi_2u)=u^{-1}\Sigma(\pi_1,\pi_2)u=1$ so that $(u^{-1}\Sigma,\pi_1u,\pi_2u)$ is a CP monoid structure on $M$ by Proposition \ref{Prop:CPalt}.

Take a triple $(\Phi,\tau_1,\tau_2)$ to the pair $(\Sigma(\tau_1,\tau_2),\Phi(\pi_1,\pi_2))$. Then
$$\Sigma(\tau_1,\tau_2)\Phi(\pi_1,\pi_2)=\Sigma(\tau_1\Phi(\pi_1,\pi_2),\tau_2\Phi(\pi_1,\pi_2))=\Sigma(\pi_1,\pi_2)=1$$
by Proposition \ref{Prop:CPalt}. If $(M,\Phi,\tau_1,\tau_2)$ is a CP monoid, then we also have that $\Phi(\pi_1,\pi_2)\Sigma(\tau_1,\tau_2)=1$, so that $\Sigma(\tau_1,\tau_2)\in M^*$.

If we start with a pair $(f,g)$ and then apply the composition of these two maps, we obtain the pair $(\Sigma(\pi_1f,\pi_2f),g\Sigma(\pi_1,\pi_2))=(f,g)$. If we start with a triple $(\Phi,\tau_1,\tau_2)$ and then apply the composition of these two maps in the opposite order, we obtain the triple $(\Phi(\pi_1,\pi_2)\Sigma,\tau_1,\tau_2)$. But, $\Phi(\pi_1,\pi_2)\Sigma(a,b)=\Phi(a,b)$ as required. So, we have a bijection as claimed.
\end{proof}

In the next section we introduce $d$-CP monoids. Theorem \ref{Thm:CPclass} has a straight-forward generalization to the case of a $d$-CP monoid for any $d\ge 0$, with only simple changes to the proof.

\section{Finite Monoids}
In this section, we show that if $M$ is a non-zero CP monoid, then \emph{every} finite monoid has an injective homomorphism into $M$. 
To facilitate the construction, we consider a generalization of a CP monoid.

\begin{definition}\label{Def:dCP} For $d\in\NN$, a \emph{categorical $d$-fold product ($d$-CP) monoid} is a monoid $M$ with distinguished elements $\tau_1,\dots ,\tau_d$ and a bijective function $\Phi:M^d\rightarrow M$ such that
$$\tau_i\Phi(m_1,\dots ,m_d)=m_i$$
for $1\le i\le d$.
\end{definition}

We could similarly define $d$-CQP monoids, but do not have occasion to use them. A $2$-CP monoid is simply a CP monoid as defined earlier. As with the original definition, this definition is motivated by the question of when a category with one object has $d$-fold products.

Note that a $0$-CP monoid is simply a monoid with one element. A $1$-CP monoid is a monoid with a specified unit $\tau\in M$; $\Phi$ must be given by $\Phi(m)=\tau^{-1}m$. Thus, any monoid may be endowed with the structure of a $1$-CP monoid. While this is easy to check directly, this may also be viewed as the $d=1$ case of the appropriate generalization of Theorem \ref{Thm:CPclass}, where $(M,\text{id},1)$ is the canonical $1$-CP structure on $M$.

As we explained in Section \ref{Section:CQPMonoids}, the arguments in this paper have straight-forward generalizations to $d$-CP (or $d$-CQP) monoids when $d\ge 2$.
One could go further. In the definition of a $d$-CP monoid, one might let $d$ be an inifinite cardinal. However, many of the arguments in this paper require $d$ to be finite.

We will use the following properties of a $d$-CP monoid $M$ whose proofs are easy generalizations of the proofs for the $d=2$ case:
\begin{enumerate}
\item $\Phi(\tau_1,\dots ,\tau_d)=1$. See Proposition \ref{Prop:CPalt}.
\item $\Phi(m_1,\dots ,m_d)n=\Phi(m_1n,\dots ,m_dn)$ for all $m_1,\dots ,m_d,n\in M$. See Proposition \ref{Prop:CPalt}.
\item If $M\ne 0$, then $\tau_1,\dots ,\tau_d$ are distinct. See Proposition \ref{Prop:Free}.
\end{enumerate}

$d$-CP products have the following property that we did not investigate for $2$-fold products earlier.

\begin{proposition}\label{Prop:Anti} Let $M$ be a non-zero $d$-CP monoid with distinguished elements $\tau_1,\dots ,\tau_d$ and bijection $\Phi:M^d\rightarrow M$. Let $E_d$ be the monoid of functions from the set $\{1,\dots ,d\}$ to itself. Then the function $\phi:E_d\rightarrow M$ given by $\phi(f)=\Phi(\tau_{f(1)},\dots ,\tau_{f(d)})$ is an injective monoid antihomomorphism.
\end{proposition}

\begin{proof} Note first that $\tau_i\phi(f)=\tau_{f(i)}$. So, if $f\ne g$ then choose an $i$ such that $f(i)\ne g(i)$. Then $\tau_i\phi(f)\ne\tau_i\phi(g)$, so that $\phi(f)\ne\phi(g)$ and $\phi$ is injective.

To see that $\phi$ is an antihomomorphism, consider
\[\begin{array}{clllc}\phi(f)\phi(g) &= &\Phi(\tau_{f(1)},\dots ,\tau_{f(d)})\Phi(\tau_{g(1)},\dots ,\tau_{g(d)}) &= &\\
&& \Phi\left(\tau_{f(1)}\Phi(\tau_{g(1)},\dots ,\tau_{g(d)}),\dots ,\tau_{f(d)}\Phi(\tau_{g(1)},\dots ,\tau_{g(d)})\right) &= &\\
&& \Phi(\tau_{g(f(1))},\dots ,\tau_{g(f(d))}) &= &\phi(gf).\end{array}\]
 Also,
$\phi(\text{id})=\Phi(\tau_1,\dots ,\tau_d)=1$.
This shows that $\phi$ is an antihomomorphism as claimed. 
\end{proof}

\begin{corollary}\label{Cor:FiniteUnits} Let $M$ be a non-zero $d$-CP monoid with distinguished elements $\tau_1,\dots ,\tau_d\in M$ and bijection $\Phi:M^d\rightarrow M$ and let $S_d$ be the symmetric group on ${1,\dots ,d}$. Then the function $\psi:S_d\rightarrow M$ given by $\psi(\sigma)=\Phi(\tau_{\sigma^{-1}(1)},\dots ,\tau_{\sigma^{-1}(d)})$ is an injective monoid homomorphism.
\end{corollary}

\begin{proof} $\psi=\phi |_{S_d}\circ\iota$ where $\phi$ is as in the proposition and $\iota:S_d\rightarrow S_d$ is the antiautomorphism given by $\iota(\sigma)=\sigma^{-1}$.
\end{proof}

To further leverage the antihomomorphism in the proposition above, we show how to build higher order $d$-CP structures on a monoid from lower order ones in the following lemma, which is a specialization of a standard construction for combining products in category theory.

\begin{lemma}\label{Lemma:Combine} Let $M$ be a monoid that is a $d$-CP monoid in the following ways:
\begin{enumerate}
\item $(M,\tau_{i1},\dots ,\tau_{id_i},\Phi_i)$ is a $d_i$-CP monoid for $1\le i\le s$ and
\item $(M,\rho_1,\dots ,\rho_s,\Psi)$ is an $s$-CP monoid.
\end{enumerate}
Then for $D=\sum_{i=1}^sd_i$, $M$ is a $D$-CP monoid with
$$\tau_{11}\rho_1,\dots ,\tau_{1d_1}\rho_1,\tau_{21}\rho_2,\dots ,\tau_{sd_s}\rho_s\in M$$
and bijection $\Delta:M^D\rightarrow M$ given by
$$\Delta(m_{11},\dots ,m_{1d_1},m_{21},\dots ,m_{sd_s}) =$$
$$\Psi(\Phi_1(m_{11},\dots ,m_{1d_1}),\dots ,\Phi_s(m_{s1},\dots ,m_{sd_s})).$$
\end{lemma}

\begin{proof} Straight-forward.\end{proof}

\begin{corollary} Let $M$ be a $d$-CP monoid. Then $M$ is also a $d^2$-CP monoid. If $d=2$, then $M$ is also a $e$-CP monoid for every $e\ge 1$.\end{corollary}

\begin{proof} For the first statement, let $s=d$ and $d_1,\dots ,d_d=d$.

For the second statement, any monoid is a $1$-CP monoid with structure $(M,\text{id},1)$. For $e>1$, use $s=2$, $d_1=e-1$ and $d_2=1$ in the lemma, so that the statement follows by induction. \end{proof}

\begin{theorem}\label{Thm:Finite} Let $M$ be a non-zero $d$-CP monoid where $d\ge 2$ and let $N$ be any finite monoid. Then there is an injective monoid homomorphism from $N$ into $M$.\end{theorem}

\begin{proof} Let $E$ be the monoid of (set) functions from $N$ to itself. Then we have an injective antihomomorphism $N\rightarrow E$ taking $n$ to $\rho_n$ where $\rho_n(x)=xn$. Repeatedly applying the previous corollary, we have that $M$ is a $D$-CP monoid for some $D\ge\# N$. (In the $d=2$ case, we could take $D=\#N$.) Choose a set $X$ such that $N\subseteq X$ and $X$ has $D$ elements. Let $E^\prime$ be the monoid of set functions from $X$ to $X$. Then we have an injective monoid homomorphism $E\rightarrow E^\prime$, taking $f$ to the function from $X$ to $X$ that agrees with $f$ on $N\subseteq X$, and is the identity on the rest of $X$. Identifying the elements of $X$ with the integers $1,\dots ,D$ in some arbitrary manner, we obtain from Proposition \ref{Prop:Anti} an injective monoid antihomomorphism $E^\prime\rightarrow M$. The composition
$$N\hookrightarrow E\hookrightarrow E^\prime\hookrightarrow M$$
of two injective monoid antihomomorphisms and one injective monoid homomorphism is the desired injective monoid homomorphism.
\end{proof}

\begin{corollary} Let $X$ be an object of a category $\calC$ such that $X$ is $X^d$ for some $d\ge 2$ and $X$ has some endomorphism other than the identity. If $N$ is any finite monoid, then $N$ has a faithful action on $X$.
\end{corollary}
\begin{proof} That $X=X^d$ implies that $M=\End(X)$ is a $d$-CP monoid. That $X$ has a nontrivial endomorphism implies that $M\ne 0$. Thus, the previous theorem gives an injective monoid homomorphism $N\rightarrow M$, which is to say a faithful action of $N$ on $X$.
\end{proof}

\section{A Conjecture and Two Questions}

We pose some questions concerning generation of CQP monoids. While we have the universal CP monoid $U$ in mind, we make the following conjecture in greater generality.

\begin{conjecture}\label{Conj:NotFG} If $M$ is a non-zero CQP monoid, then $M$ is not finitely generated as a monoid.\end{conjecture}

If $M$ is a CP monoid, then by Theorem \ref{Thm:Finite}, $M$ contains a copy of every finite monoid $N$. Thus, $M$ is very large indeed and we take this as evidence for the conjecture at least for CP monoids and those CQP monoids such that $M/\!\!\sim_{cp}$ is not trivial. However, there are only countably many finite monoids up to isomorphism, so it is possible to construct a monoid with just $2$ generators that contains a copy of every finite monoid. While we take Theorem \ref{Thm:Finite} as evidence for Conjecture \ref{Conj:NotFG}, it does not immediately give a proof.

Note that if $(M,\Sigma, \pi_1,\pi_2)$ is a CQP monoid generated as a monoid by $A_1,\dots , A_n$, then by repeatedly combining the $A_i$ in pairs through $\Sigma$, we can obtain a single $A\in M$ such that $w_iA=A_i$ for suitable words $w_i$ in $\pi_1$ and $\pi_2$. Thus, if $M$ is a finitely generated CQP monoid, then $M$ is generated by three elements $\pi_1,\pi_2,A$ for a suitable $A\in M$.

Assuming that the conjecture is correct, then there should still be good infinite generating sets for $U$ and for $U^*$. We consider a candidate.

For a CP monoid $(M,\Sigma,\pi_1,\pi_2)$, let $\mathbb{S}(M)$ be the set of $d$-CP monoid structures that can be recursively built using Lemma \ref{Lemma:Combine} starting from the 2-CP monoid structure $(\Sigma,\pi_1,\pi_2)$ and the 1-CP monoid structure $(\text{id},1)$. For example, the 3-CP monoid structures in $\mathbb{S}(M)$ are $(\Sigma(\Sigma\times\text{id}),\pi_1^2,\pi_2\pi_1,\pi_2)$ and $(\Sigma(\text{id}\times\Sigma),\pi_1,\pi_1\pi_2,\pi_2^2)$.

\begin{question} Let $E_d$ be the monoid of functions from $\{1,\dots ,d\}$ to itself. Is the universal CP monoid $U$ generated by the images of $E_d$, $2\le d<\infty$, under the antihomomorphisms determined by $(\Phi,\tau_1,\cdots ,\tau_d)\in\mathbb{S}(U)$ as in Proposition \ref{Prop:Anti}?
\end{question}

\begin{question} Is the unit group $U^*$ of the universal CP monoid $U$ generated by the images of $S_d$, $2\le d<\infty$, under the homomorphisms determined by $(\Phi,\tau_1,\cdots ,\tau_d)\in\mathbb{S}(U)$ as in Corollary \ref{Cor:FiniteUnits}?
\end{question}

Of course, the image of an element of $S_d$ is a unit of finite order. However, the subgroup of $U^*$ generated by these images does have elements of infinite order, such as

\[\xymatrix{
&&\circ\ar@{-}[ld]\ar@{-}[rd]&&&\circ\ar@{-}[ld]\ar@{-}[rd]&&&&\circ\ar@{-}[ld]\ar@{-}[rd]&\\
&\circ\ar@{-}[ld]\ar@{-}[rd]&&\pi_2\pi_1&\pi_2&&\pi_1&=&\circ\ar@{-}[ld]\ar@{-}[rd]&&\pi_2^2.\\
\pi_2&&\pi_1^2&&&&&\pi_1&&\pi_1\pi_2&
}\]

\end{document}